\documentclass[oneside,11pt]{amsart}
\usepackage{amsmath}
\usepackage{amssymb}

\newtheorem{thm}{Theorem}[section]
\newtheorem{prop}[thm]{Proposition}
\newtheorem{lem}[thm]{Lemma}

\newtheorem*{main}{Main Theorem}

\theoremstyle{definition}
\newtheorem{defn}[thm]{Definition}
\newtheorem{rem}[thm]{Remark}

\renewcommand{\bar}[1]{\overline{#1}}

\newcommand{\set}[2]{\{\,{#1} \mid {#2} \,\}}

\renewcommand{\emptyset}{\varnothing}

\newcommand{\field}[1]{\mathbb{#1}}

\newcommand{\PP}{\field{P}}
\newcommand{\HH}{\field{H}}
\newcommand{\EE}{\field{E}}

\DeclareMathOperator{\CAT}{CAT}

\DeclareMathOperator{\diam}{diam}
\DeclareMathOperator{\Sat}{Sat} %Saturation
\DeclareMathOperator{\Hull}{Hull}

\begin{document}

\title[Various boundaries of relatively hyperbolic groups]{Relations between various boundaries of relatively hyperbolic groups}

\author{Hung Cong Tran}
\address{Dept.\ of Mathematical Sciences\\
University of Wisconsin--Milwaukee\\
P.O.~Box 413\\
Milwaukee, WI 53201\\
USA}
\email{hctran@uwm.edu}

\date{\today}

\begin{abstract}
Suppose a group $G$ is relatively hyperbolic with respect to a collection $\PP$ of its subgroups and also acts properly, cocompactly on a $\CAT(0)$ (or $\delta$--hyperbolic) space $X$. The relatively hyperbolic structure provides a relative boundary $\partial(G,\PP)$. The $\CAT(0)$ structure provides a different boundary at infinity $\partial X$. In this article, we examine the connection between these two spaces at infinity. In particular, we show that $\partial (G,\PP)$ is $G$--equivariantly homeomorphic to the space obtained from $\partial X$ by identifying the peripheral limit points of the same type.
\end{abstract}

 % Geometric group theory
\maketitle

\section{Introduction}

In \cite{MR919829}, Gromov introduced relatively hyperbolic groups. This concept was studied by Bowditch \cite{MR2922380}, Osin \cite{Osin06}, Farb \cite{MR1650094} and others. Many equivalent definitions for this concept were introduced. Gromov defined a relatively hyperbolic group by its action on a proper, hyperbolic space (the cusped space) (see the Definition 1 in \cite{MR2922380}). Bowditch (elaborating on an idea of Farb) viewed a relatively hyperbolic group by its action on a fine, hyperbolic graph (the coned off space) (See Definition \ref{rel}). By investigating the geometry of the cusped spaces and coned off spaces, he introduced the relative boundary of a relatively hyperbolic group. In \cite{MR1680044}, Bowditch also proved that the relative boundary of a relatively hyperbolic group is always locally connected if it is connected (under certain hypotheses on the peripheral subgroups).

Many $\CAT(0)$ or $\delta$--hyperbolic spaces admit proper, cocompact actions of relatively hyperbolic groups. Such spaces were studied by Hruska--Kleiner \cite{MR2175151}, Caprace \cite{MR2665193}, Kapovich--Leeb \cite{MR1339818}, Hindawi \cite{MR2172938} and others. Suppose a group $G$ is relatively hyperbolic with respect to a collection $\PP$ of its subgroups and also acts properly, cocompactly on a $\CAT(0)$ (or $\delta$--hyperbolic) space $X$. What is the connection between the boundary of $X$ and the relative boundary of $G$? The main theorem in this paper answers this question.

\begin{main}
Let $(G,S,\PP)$ be a relatively hyperbolic group that acts properly and cocompactly on a $\CAT(0)$ space or a $\delta$--hyperbolic space $X$. Then the relative boundary $\partial (G,\PP)$ of $G$ is $G$--equivariantly homeomorphic to the space obtained from $\partial X$ by identifying the peripheral limit points of the same type. 
\end{main} 

We remark that the case of $\delta$--hyperbolic spaces was known before. It follows from the work of Gerasimov \cite {MR2989436} and Gerasimov--Potyagailo \cite{GerasimovPotyagailo} on Floyd boundaries, using Gromov's result that the Floyd boundary is the same as the Gromov boundary for word hyperbolic groups (\cite {MR919829}, 7.2.M) (For an alternate proof, see \cite{MatsudaOguniYamagata}). However, the case of $\CAT(0)$ spaces was still open and it is solved by the main theorem of this paper. Moreover, the technique to prove the main theorem in the case of $\CAT(0)$ spaces is different from the technique used in the case of $\delta$--hyperbolic spaces because of the lack of hyperbolicity. However, the technique we use for proving the case of $\CAT(0)$ spaces could be applied to the case of $\delta$--hyperbolic spaces with slight modification. 

The challenge here is to build a $G$--equivariant quotient map $f$ from $\partial X$ to $\partial (G,\PP)$ by investigating the relations between quasigeodesic rays in the Cayley graph and geodesic rays in the coned off Cayley graph. In \cite{Hruska10}, Hruska found relations between quasigeodesic segments in the Cayley graph and geodesic segments in the coned off Cayley graph to study the quasiconvexity of subgroups. However, the relations between quasigeodesic rays in the Cayley graph and geodesic rays in the coned off Cayley graph require us to examine many cases and they are not immediate corollaries of Hruska's result. Moreover, the proof of continuity of $f$ is not quite obvious since the topologies on $\partial X$ and $\partial (G,\PP)$ are not equipped in similar ways.

We now take a look on a particular example to visualize the theorem. Let $G$ be the fundamental group of a complete $3$--dimensional hyperbolic manifold of finite volume. The group $G$ is hyperbolic relative to the collection $\PP$ of its cusp subgroups and $G$ acts on the hyperbolic space $\HH^3$. Moreover, the action satisfies all conditions in Definition 1 of \cite{MR2922380}. Thus, the relative boundary of $(G,\PP$) is obviously $\partial(G,\PP)=\partial \HH^3=S^2$. Ruane \cite{MR2182938} considered a metric space $X$ obtained from the hyperbolic space $\HH^3$ by deleting a family of disjoint open horoballs of the action and endowing $X$ with the induced length metric. The space $X$ was known to be a complete $\CAT(0)$ space and the boundary of $X$ is the Sierpinski carpet. We obtain $S^2$ by identifying each of the circles inside the Sierpinski carpet and the circle of the outside boundary of the Sierpinski carpet. Thus, we can see the connection between the relative boundary of the group $G$ and the $\CAT(0)$ boundary of the space $X$. 

From the main theorem, we could see that we are able to understand the relative boundary of a relatively hyperbolic group if we have enough information about the boundary of the space on which the group acts properly, cocompactly and the peripheral limit points of each peripheral left coset. We take a look at some examples:

Let $G$ be a $\delta$--hyperbolic group and $\PP$ be a finite malnormal collection of quasiconvex subgroups. Thus, $G$ is not only a hyperbolic group but also relatively hyperbolic with respect to the collection $\PP$ (see \cite {MR2922380}). By the main theorem, the relative boundary $\partial (G,\PP)$ is obtained from the boundary of $G$ by identifying the limits points of each peripheral left coset.

In particular, if $G$ is the fundamental group of a hyperbolic surface $M$ and $H$ is the fundamental group of a closed essential curve $C$ of $M$, then the group $G$ is not only a hyperbolic group but also relatively hyperbolic with respect to $H$. We know that the boundary of $G$ is a circle $S^1$ and each peripheral left coset gives us a pair of limit points. By identifying each such pair, we see that the relative boundary is tree-graded over circles (for the definition of tree graded spaces, see \cite{MR2153979}). This example was also discussed by Bowditch in \cite{MR1837220}. He examined these boundaries from a different point of view using the structure of the splitting as an amalgam.

In \cite{MR2175151}, Hruska and Kleiner proved that if a group $G$ is relatively hyperbolic with respect to a collection of virtually abelian subgroups of rank at least two and $G$ acts properly and cocompactly on a $\CAT(0)$ space $X$, then $X$ has isolated flats. Moreover, the set of all peripheral limit points is also the set of all points that lie in the boundaries of these flats. Two peripheral points have the same type iff they lie in the boundary of the same isolated flat. We know that each boundary of an isolated flat is a sphere in the boundary of $X$. Thus, by identifying each such sphere from the boundary of $X$ we obtained the relative boundary of $G$. 

Many relatively hyperbolic groups are also ``CAT(0) groups'' in the sense that they act properly and cocompactly on CAT(0) spaces. For instances, the fundamental groups of compact, irreducible 3-manifolds with at least one hyperbolic JSJ component \cite{MR1339818}, CAT(0) groups with isolated flats \cite{MR2175151}, relatively hyperbolic Coxeter groups \cite{MR2665193}, and fundamental groups of certain real analytic, nonpositively curved 4-manifolds \cite{MR2172938} and others. The main theorem could be applied to such groups in order to understand the connection between the two different kinds of boundaries at infinity of them.

Another advantage of the theorem is that it helps us to see the connection among the boundaries of the cusped space $Y$, the coned off space $K$ and the space $X$ the group acts properly, cocompactly on. (In particular, $X$ carries the intrinsic geometry of the group.) In \cite{MR2922380}, Bowditch showed the connection between the cusped space $Y$ and the coned off space $K$ as the following: \[\partial Y\simeq \Delta_{\infty} K.\] where $\Delta_{\infty}K$ is a topological space that contains $\partial K$, we will define $\Delta_{\infty}K$ later (see Section 3). In particular, there is a natural embedding of $\partial K$ into $\partial Y$. Bowditch used $\partial Y$ or $\Delta_{\infty}K$ to define the relative boundary of a relatively hyperbolic group $G$ and the main theorem deduces that there is a quotient map from $\partial X$ to the $\partial Y$. As a consequence, there is an embedding from $\partial K$ into $\partial X$. Therefore, we have the complete connection among the boundaries of three such kinds of spaces. 

In this paper, we only focus on the case of $\CAT(0)$ spaces. The proof for the case of $\delta$--hyperbolic spaces is nearly identical and we leave it to the reader. 

\subsection*{Acknowledgments}
I would like to thank my advisor Prof.~Christopher Hruska for very helpful comments and suggestions. I also thank the referee for advice that improved the exposition of the paper and Shin-ichi Oguni for helpful correspondence.  

\section{Some properties of $\delta$--hyperbolic spaces and $\CAT(0)$ spaces}
\label{sec:Properties}

In this section, we review the concept of $\CAT(0)$ spaces and $\delta$--hyperbolic spaces and their well-known properties we need to prove the main theorem. Most of information in this section are cited from \cite{MR1377265} and \cite{MR1086648}. 
%Section~\ref{sec:Properties} 
\begin{defn}
We say that a geodesic triangle $\Delta$ in a geodesic space X satisfies \emph{the $\CAT(0)$ inequality} if $d(x,y)\leq d(\bar{x},\bar{y})$ for all points $x,y$ on the edges of $\Delta$ and the corresponding points $\bar{x},\bar{y}$ on the edges of the comparison triangle $\bar{\Delta}$ in Euclidean space $\EE^2$.
\end{defn}

\begin{defn}
\label{CAT(0)}
A geodesic space X is said to be a \emph{CAT(0) space} if every triangle in X satisfies the $\CAT(0)$ inequality.

If $X$ is a $\CAT(0)$ space, then \emph{the CAT(0) boundary} of $X$, denoted $\partial X$, is defined to be the set of all equivalence classes of geodesic rays in X, where two rays $c,c'$ are equivalent if the Hausdorff distance between them is finite.

On $\partial X$, we could build a topology as follows :

We note that for any $x\in X$ and $\xi \in \partial X$ there is a unique geodesic ray $\alpha_{x,\xi}\colon [0,\infty)\to X$ with $\alpha_{x,\xi}(0)=x$ and $[\alpha_{x,\xi}]=\xi$. We have a topology on $\partial X$ by using the sets $U(x,\xi,R,\epsilon)=\set{\xi'\in\partial X}{d\bigl(\alpha_{x,\xi}(R),\alpha_{x,\xi'}(R)\bigr)\leq \epsilon}$, where $x\in X,\xi\in\partial X, R>0$ and $\epsilon>0$ as a basis.
\end{defn}

\begin{defn}
\label{defn:HyperbolicSpace}
A geodesic metric space $(X,d)$ is \emph{$\delta$--hyperbolic} if every geodesic triangle with vertices in $X$ is \emph{$\delta$--thin} in the sense that each side lies in the $\delta$--neighborhood of the union of other sides.
If $X$ is a $\delta$--hyperbolic space, then we could build \emph{the hyperbolic boundary} of $X$, denoted $\partial X$, in the same way as for a $\CAT(0)$ space. That is, the hyperbolic boundary of $X$ is defined to be the set of all equivalence classes of geodesic rays in X, where two rays $c,c'$ are equivalent if the Hausdorff distance between them is finite. However, the topology on it is slightly different from the topology on the boundary of a $\CAT(0)$ space. (see \cite{MR1744486} for details.)
\end{defn}

\begin{rem}
For each finite path $\alpha$ in a space $X$, we denote the endpoints of $\alpha$ by $\alpha_+$ and $\alpha_-$. For each ray $\alpha$ in a space $X$, we denote the initial point of $\alpha$ by $\alpha_+$.
\end{rem}

\begin{defn}
Let $(X,d)$ be a metric space.
\begin{enumerate}
\item A path $p$ in $X$ is an \emph{$(L,C)$--quasigeodesic} for some $L\geq 1, C\geq 0$, if for every subpath $q$ of $p$ the inequality $\ell(q)\leq Ld(q_+,q_-)+C$ holds.
\item A path $p$ in $X$ is a \emph{quasigeodesic} if it is an $(L,C)$--quasigeodesic for some $L\geq 1, C\geq 0$. 
\item A path $p$ in $X$ is an \emph{L--quasigeodesic} if it is an $(L,L)$--quasigeodesic for some $L\geq 1$. 
\item Two quasigeodeics are \emph{equivalent} if the Hausdorff distance between them is finite.
\end{enumerate}
\end{defn}

\begin{defn}
Let $X$ and $Y$ be two metric spaces. 
\begin{enumerate}
\item A function $\Phi$ from X to Y is a \emph{K--quasi-isometry} for some constant $K\geq 1$ if for all $x,x'\in X$ the inequality \[\frac{1}{K}d_X(x,x')-1\leq d_Y\bigl(\Phi(x),\Phi(x')\bigr)\leq Kd_X(x,x')+K.\] holds and $N_K\bigl(\Phi(X)\bigr)=Y$.
\item A function $\Phi$ from X to Y is a \emph{quasi-isometry} if it is a $K$--quasi-isometry for some $K\geq 1$.
\end{enumerate}
\end{defn}

Here are some well-known properties of $\delta$--hyperbolic spaces and $\CAT(0)$ spaces. 
\begin{lem}
\label{l6}
Let $\alpha$ and $\beta$ be two geodesics in a $\CAT(0)$ space X such that the endpoints of $\beta$ lie in $N_{\epsilon}(\alpha)$ for some positive number $\epsilon$. Then $\beta\subset N_{\epsilon}(\alpha)$. 
\end{lem}

\begin{rem}
\label{l30}
For the case of $\delta$--hyperbolic spaces, we have a similar result:

For each $\epsilon >0$, $\delta >0$ there is $A=A(\epsilon, \delta) >0$ such that the following holds. Let $\alpha$ and $\beta$ be two geodesics in a $\delta$--hyperbolic space X such that the endpoints of $\beta$ lie in $N_{\epsilon}(\alpha)$. Then $\beta\subset N_A(\alpha)$.
\end{rem}

\begin{lem}
\label{l14}
For each choice of positive constants $\delta$ and $\sigma$, there is a positive number $R=R(\delta,\sigma)$ such that the following holds. Let $\alpha$ and $\alpha'$ be two equivalent geodesic rays in a $\delta$--hyperbolic space such that $d(\alpha_+,\alpha'_+)\leq\sigma$ or $\alpha$ and $\alpha'$ be two geodesic segments such that $d(\alpha_+,\alpha'_+)\leq\sigma$ and $d(\alpha_-,\alpha'_-)\leq\sigma$. Then the Hausdorff distance between them is at most $R$.
\end{lem}
 
\begin{lem}
\label{l1}
Let $X$ be a $\delta$--hyperbolic space. There is a number $M=M(\delta)$ such that the following holds. Let $\alpha$ and $\alpha'$ be two equivalent geodesic rays in $X$ with the same initial point $z$. Let $x$ and $y$ be two points in $\alpha$ and $\alpha'$ respectively such that $d(z,x)=d(z,y)$. Then $d(x,y)\leq M$.
\end{lem} 
 
\section{ Some properties of fine graphs and hyperbolic graphs}

In this section, we review some concepts related to graphs, hyperbolic and fine graphs, the construction of the hyperbolic closure and infinite hyperbolic closure of a hyperbolic and fine graph. Most of the information in this section is cited from \cite{MR2922380}.
\begin{defn} 
Let $K$ be a graph with the vertex set $V(K)$ and the edge set $E(K)$. We write $V_0(K)$ and $V_{\infty}(K)$ respectively for the sets of vertices of finite and infinite degree. A \emph{path} of length $n$ connecting $x,y\in V$ is a sequence, $x_0x_1\cdots x_n$ of vertices, with $x_0=x$ and $x_n=y$, and with each $x_i$ equal to or adjacent to $x_{i+1}$. A \emph{ray} $x_0x_1\cdots x_n\cdots$ is defined similarly. A path $x_0x_1\cdots x_n$ is an $arc$ if the $x_i$ are all distinct. A $cycle$ is a closed path $x_0=x_n$, and a $circuit$ is a cycle with all vertices distinct. We regard two cycles as the same if their vertices are cyclically permuted. We frequently regard arcs and circuits as subgraphs of $K$. We put a $path$ $metric$, $d_K$, on $V(K)$, where $d_K(x,y)$ is the length of the shortest path in $K$ connecting $x$ and $y$.
\end{defn}

\begin{lem} [Proposition 2.1, \cite{MR2922380}] 
\label{l16}
Let K be a graph. The following are equivalent:
\begin{enumerate}
\item For each positive number $n$, each edge of $K$ is contained in only finitely many circuits of length $n$.
\item For each positive number $n$ and $x,y\in V(K)$, the set of arcs of length $n$ connecting $x$ to $y$ is finite.
\end{enumerate}
\end{lem}

\begin{defn}
 We say that a graph is \emph{fine} if it satisfies the properties in Lemma \ref{l16}.
\end{defn} 
 
From now, we assume that $K$ is connected, fine and hyperbolic. Let $\partial K$ be the usual hyperbolic boundary of $K$ and we assume that $\partial K$ has more than two points. We define $\Delta K= V(K)\cup\partial K$, we call \emph{the hyperbolic closure of K} and $\Delta_{\infty} K= V_{\infty}(K)\cup\partial K$, we call \emph{the infinite hyperbolic closure of K}. We put a topology on $\Delta(K)$ as follows:

For each $a\in\Delta K$ and $A$ a finite set of $V(K)$ that does not contain $a$, we define $M(a,A)$ to be the set of points $b\in\Delta K$ such that there is at least one geodesic $\alpha$ from $a$ to $b$ that does not meet $A$.
We define a set $U\subset \Delta K$ to be open if for all $a\in U$, there is a finite subset $A\subset V(K)$ that does not contain $a$ such that $M(a,A)\subset U$.

We observe that the set $V_0$ of all finite degree vertices of K is also the set of all isolated points in $\Delta K$. Thus, we could view $\Delta_{\infty} K$ as $\Delta K$ minus the set of all isolated points.

\begin{defn}
For each $\lambda\geq 1, c\geq 0$, $a\in\Delta K$ and a finite set $A$ of $V(K)$ that does not contain $a$, we define $M_{(\lambda,c)}(a,A)$ to be the set of points $b\in\Delta K$ such that there is at least one $(\lambda,c)$--quasigeodesic arc $\alpha$ from $a$ to $b$ that does not meet $A$.
\end{defn}

\begin{lem}[Bowditch, \cite{MR2922380}] Let $K$ be a fine and hyperbolic graph. Then:
\begin{enumerate}
\item For each $\lambda\geq 1,c\geq 0$, the collection of all sets of the form $M_{(\lambda,c)}(a,A)$, where $a\in\Delta K$ and $A$ is a finite set of $V(K)$ that does not contain $a$, forms the basis for the topology on $\Delta K$ as defined above.
\item $\Delta K$ is Hausdorff.
\item The subspace topology on $\partial K$ induced from $\Delta K$ agrees with the usual topology on $\partial K$.
\item $\Delta K$ and $\Delta_{\infty}K$ are compact. 
\end{enumerate}
\end{lem}

\section{ Some properties of the Cayley graph and the coned off Cayley graph of a relatively hyperbolic group}

The aim of this section is to explain the concept of relatively hyperbolic groups, the concept of the relative boundary of a relatively hyperbolic group, some properties of the Cayley graph, the coned off Cayley graph of a relatively hyperbolic group and a connection between these two graphs.

\begin{defn}
Given a finitely generated group $G$ with the Cayley graph $\Gamma(G,S)$ equipped with the path metric and a collection $\PP$ of subgroups of G, one can construct the \emph{coned off Cayley graph} $\hat{\Gamma}(G,S,\PP)$ as follows: For each left coset $gP$ where $P\in \PP$, add a vertex $v_{gP}$, we call \emph{peripheral vertex}, to the Cayley graph $\Gamma(G,S)$ and for each element $x$ of $gP$, add an edge $e(x,gP)$, we call \emph{peripheral half edge}, of length 1/2 from $x$ to the vertex $v_{gP}$. This results in a metric space that may not be proper (i.e. closed balls need not be compact).
\end{defn}

\begin{defn} [Relatively hyperbolic group]
\label{rel}
A finitely generated group $G$ is \emph{hyperbolic relative to a collection $\PP$ of subgroups of $G$} if the coned off Cayley graph is $\delta$--hyperbolic and fine.

Each group $P\in \PP$ is a \emph{peripheral subgroup} and its left cosets are \emph{peripheral left cosets} and we denote the collection of all peripheral left cosets by $\Pi$. 
\end{defn}

\begin{defn}[Relative boundary]
Suppose $(G,\PP)$ is relatively hyperbolic, and $S$ is a finite generating
set for $G$. The \emph{relative boundary} $\partial(G,\PP)$ is the space
$\Delta_{\infty} \bigl( \hat\Gamma (G,S,\PP) \bigr)$.
\end{defn}

\begin{rem}
Bowditch has shown that the relative boundary does not depend on the choice of finite generating set. Indeed, if $S$ and $T$ are finite generating sets for $G$ then the spaces $\Delta_{\infty} \bigl( \hat\Gamma (G,S,\PP) \bigr)$ and $\Delta_{\infty} \bigl( \hat\Gamma (G,T,\PP) \bigr)$ are $G$--equivariantly homeomorphic (see \cite{MR2922380}).

The subset $V_{\infty} \bigl( \hat\Gamma (G,S,\PP) \bigr)$ of $\Delta_{\infty} \bigl( \hat\Gamma (G,S,\PP) \bigr)$ consists of the peripheral vertices whose associated peripheral left cosets are infinite which we call the parabolic points. 
\end{rem}

\begin{rem}
In $\hat{\Gamma}(G,S,\PP)$, we call an edge labelled by an element in the set $S$ of generators a \emph{S--edge} and we call an edge that consists of two peripheral half edges with some peripheral vertex $v_{gP}$ in the middle a \emph{peripheral edge}.

From now, we denote the metric in the $\Gamma(G,S)$ by $d_S$ and the metric in $\hat{\Gamma}(G,S,\PP)$ by $d$.
\end{rem}

\begin{lem}
[Osin, \cite{Osin06}]If $G$ is a finitely generated group which is hyperbolic relative to the collection $\PP$ of subgroups of $G$, then $\PP$ is finite. 
\end{lem}

\begin{defn}\label{d1}
Let ($X,d)$ be a geodesic metric space. Let $\bigl(Y(p)\bigr)_{p\in P}$ be a collection of subsets of $X$, indexed by a set $P$. We say that:
\begin{enumerate}
\item $\bigl(Y(p)\bigr)_{p\in P}$ is \emph{quasidense} if $X=N_t\bigl(\bigcup_{p\in P}Y(p)\bigr)$ for some $t\geq 0$.
\item $\bigl(Y(p)\bigr)_{p\in P}$ is \emph{locally finite} if $\set{p\in P}{d\bigl(x,Y(p)\bigr)\leq u}$ is finite for all $x\in X$ and $u\geq 0$.
\item $\bigl(Y(p)\bigr)_{p\in P}$ has \emph{the bounded penetration property} if, given any $r\geq 0$, there is some $D=D(r)\geq 0$ such that for all distinct $p,q\in P$, the set $N_r\bigl(Y(p)\bigr)\cap N_r\bigl(Y(q)\bigr)$ has diameter at most $D$.
\item $\bigl(Y(p)\bigr)_{p\in P}$ has \emph{the uniform neighborhood quasiconvexity property} if, given any $\lambda \geq 1, C\geq 0, r\geq 0$, there is some $D=D(\lambda,C,r)\geq 0$ such that any $(\lambda,C)$--quasigeodesic segment whose endpoints lie in the $r$--neighborhood of any set $Y(p)$ must lie in the $D$--neighborhood of $Y(p)$.
\end{enumerate}
\end{defn}

The proof of the following lemma is obvious, and we leave it to the reader.
\begin{lem}
The properties in Definition \ref{d1} are invariant under quasi-isometry.
\end{lem}

\begin{lem} [Dru{\c{t}}u--Sapir, \cite{MR2153979}]
Let $G$ be finitely generated and hyperbolic relative to a collection $\PP$ of subgroups of $G$. Then in $\Gamma(G,S)$, the collection of peripheral left cosets $\set{gP}{P\in\PP, g\in G}$ has the properties in Definition \ref{d1}.
\end{lem}

\begin{defn}
If $\hat{c}$ is a path in $\hat{\Gamma}(G,S,\PP)$ and $M\geq 0$, the \emph{M--saturation} of $\hat{c}$, denoted $\Sat_M(\hat{c})$, is the union of all peripheral left cosets $gP$ such that there is at least one $G$--vertex of $\hat{c}$ which lies in the $M$--neighborhood of $gP$ with respect to $d_S$. 
\end{defn}

\begin{rem}
The definition of M--saturation in this paper is slightly different than in the paper of Dru{\c{t}}u--Sapir (see \cite{MR2153979}). Dru{\c{t}}u--Sapir considered M--saturation of a path in the ordinary Cayley graph but in this paper, we consider M--saturation of a path in the coned off Cayley graph. The M--saturation concept in this paper will be used to describe the connection between some path in the ordinary Cayley graph and some path in the coned off Cayley graph (see Lemma \ref{l5} and Lemma \ref{l15}). 
\end{rem}

\begin{lem} [Theorem 3.26, \cite{Osin06}]
\label{l2} There is a positive constant $a$ such that the following holds. Let $\Delta=(p,q,r)$ be a geodesic triangle in $\hat{\Gamma}(G,S,\PP)$. Then for each $G$--vertex $v$ on $p$, there is a $G$--vertex $u$ in the union $q\cup r$ such that $d_S(u,v)\leq a$.
\end{lem}

\begin{lem}\label{l8}
Let $\pi$ and $\pi'$ be two equivalent geodesic rays in $\hat{\Gamma}(G,S,\PP)$ (i.e. the Hausdorff distance between $\pi$ and $\pi'$ is finite with respect to d) with the same initial point $h_0$. Suppose that $(g_n)$ and $(g'_n)$ are the sequences of all $G$--vertices of $\pi$ and $\pi'$ respectively. Then the Hausdorff distance between $(g_n)$ and $(g'_n)$ is finite with respect to the metric $d_S$.
\end{lem}

\begin{proof}
Let $\delta$ be the hyperbolicity constant of $\hat{\Gamma}(G,S,\PP)$, and let $M=M(\delta)$ be the constant in Lemma \ref{l1}. Let $a$ be the number in Lemma \ref{l2}. 

For each $g_n\in\pi$, choose a positive integer $m>n+ M(\delta)+a$.
Let $\pi''$ be a geodesic in $\hat{\Gamma}(G,S,\PP)$ that connects $g_m$ and $g'_m$.
By Lemma \ref{l2}, there is a $G$--vertex $u$ in the union $\pi'\cup \pi''$ such that $d_S(u,g_n)\leq a$.

If $u\in \pi''$, then $d(g_n,g_m)\leq d(g_n,u)+d(u,g_m)\leq d_S(g_n,u)+d(u,g_m)\leq a+M(\delta)$. Also, $d(g_n,g_m)=m-n$. Thus, $m-n \leq a+M(\delta)$. This is a contradiction. It implies that $u\in \pi'$. Thus, $(g_n$) lies in the $a$--neighborhood of $(g'_n$) with respect to the metric $d_S$.

Similarly, $(g'_n$) lies in the $a$--neighborhood of $(g_n$) with respect to the metric $d_S$. Therefore, the Hausdorff distance between $(g_n)$ and $(g'_n)$ is finite with respect to the metric $d_S$.
\end{proof}

\begin{lem} [Lemma 8.8, \cite{Hruska10}]
\label{l3} For each $\epsilon >0$ there is $A=A(\epsilon) >0$ such that the following holds. Let c be an $\epsilon$--quasigeodesic in $\Gamma(G,S)$ and $\hat{c}$ be a geodesic in $\hat{\Gamma}(G,S,\PP)$ with the same endpoints in $G$. Then each $G$--vertex of $\hat{c}$ lies in the $A$--neighborhood of some vertex of c with respect to the metric $d_S$.
\end{lem}

\begin{defn}
Let $c$ be a quasigeodesic in $\Gamma(G,S)$. If $c_0$ is any subset of $c$, the \emph{Hull} of $c_0$ in $c$, denoted $\Hull_c(c_0)$ is the smallest connected subspace of $c$ containing $c_0$.
\end{defn}

\begin{lem} \label{l5}
For each $\epsilon>0$ there are constants $A=A(\epsilon)>0$ and $B=B(\epsilon)>0$ such that the following holds. Let $c$ be an $\epsilon$--quasigeodesic ray in $\Gamma(G,S)$ such that $c$ is not contained in $N_R(gP)$ for any peripheral left coset $gP$ and any positive number $R$. Then there is a geodesic ray $\hat{c}$ in $\hat{\Gamma}(G,S,\PP)$ such that $\hat{c}$ and $c$ have the same initial point, each $G$--vertex of $\hat{c}$ lies in the $A$--neighborhood of $c$ and $c$ lies in the $B$--neighborhood of the $\Sat_0(\hat{c})$ with respect to the metric $d_S$.
\end{lem}

\begin{proof}
Let $A=A(\epsilon)$ be the constant from Lemma \ref{l3}. Since $(gP)_{{gP}\in \Pi}$ is a uniform neighborhood quasiconvex collection of subsets of $\Gamma(G,S)$, then there is some $A_1=A_1(A,\epsilon)\geq 0$ such that any $\epsilon$--quasigeodesic segment whose endpoints lie in the $A$--neighborhood of any set $gP$ must lie in the $A_1$--neighborhood of $gP$.

Choose $(z_n)$ in c such that $z_0$ is the initial endpoint of $c$ and $d_S(z_0,z_n)\rightarrow \infty$. For each $n$, choose a geodesic $\hat{c}_n$ in $\hat{\Gamma}(G,S,\PP)$ with the endpoints $z_0$ and $z_n$.
 
For each $m\geq 0$ and $n>m$, each $G$--vertex of $\hat{c}_n$ lies in the $A$--neighborhood of some vertex of [$z_0,z_n$]. For each $G$--vertex $v$ of $\hat{c}_n$, let \[c_v=\Hull_{[z_0,z_n]}\bigl([z_0,z_n]\cap B(v,A)\bigr).\] For each edge $e$ of $\hat{c}_n$ with $G$--endpoints $v$ and $w$, let \[c_e= \Hull_{[z_0,z_n]}(c_v\cup c_w).\] We note that $e$ is either an $S$--edge in $\Gamma(G,S)$ or $e$ is a peripheral edge that consists of two peripheral half edges with some $v_{gP}$ in the middle. We see that $[z_0,z_n]$ is covered by the sets $c_e$ for all edge $e$ of $\hat{c}_n$. Thus, $z_m\in c_e$ for some edge $e$ of $\hat{c}_n$.

If $e$ is an $S$--edge, the length $c_e$ is at most $\epsilon(2A+1)+\epsilon$. It follows that $z_m$ lies within a distance $\epsilon(2A+1)+\epsilon$ of $[z_0,z_n]\cap B(v,A)$, where $v$ is a $G$--vertex incident to $e$. Thus, $z_m$ lies within $\epsilon(2A+1)+\epsilon+A$ of $v$. We choose $v_{m,n}=v$ and we have $d_S(v_{m,n}, z_m) \leq\epsilon(2A+1)+\epsilon+A$.

If $e$ is a peripheral edge, then the midpoint of $e$ is $v_{gP}$ for some peripheral left coset $gP$ and $[z_0,z_n]\cap B(v,A)$ lies in the $A$--neighborhood of $gP$ for each $G$--endpoint $v$ of $e$. Since each point in $c_e$ lies between two such points, $c_e$ lies in the $A_1$--neighborhood of $gP$. In particular, $d_S(z_m,gP)\leq A_1$. We choose $v_{m,n}=v_{gP}$.

In both cases, we see that there are only finitely many possibilities for $v_{m,n}$ when we fix m and let $n>m$ be arbitrary.
By a diagonal sequence argument, we can suppose that $v_{m,n}=v_m$ is independent of $n>m$. In other words, $v_m\in\hat{c}_n$ for all $n>m$.

For each $0\leq m<n$ there is a geodesic $d_{m,n}$ from $z_0$ to $v_m$ such that $d_{m,n}\subset\hat{c}_n$.
Since $\hat{\Gamma}(G,S,\PP)$ is fine, then there are only finitely many possibilities for $d_{m,n}$ when we fix m and let $n>m$ be arbitrary. By a diagonal sequence argument again, we can suppose that $d_{m,n}=d_m$ is independent of $n>m$. In other words, $d_m\subset\hat{c}_n$ for all $n>m$.

We claim that the set $\set{v_m}{m\geq 0}$ is infinite. Suppose that $\set{v_m}{m\geq 0}$ is finite. For each $m\geq 0$, if $v_m$ is a $G$--vertex, then we choose some peripheral left coset $g_mP_m$ containing $v_m$, otherwise we choose $g_mP_m=gP$ where $v_m =v_{gP}$. The set $\set{g_mP_m}{m\geq 0}$ is also finite. Also, \[d_S(z_m,g_mP_m)\leq \max\bigl\{\epsilon(2A+1)+\epsilon+A, A_1\bigr\}.\] Thus, we could find a subsequence $(m_n)$ and $gP\in \set{g_mP_m}{m\geq 0}$ such that $g_{m_n}P_{m_n} =gP$. It follows that \[d_S(z_{m_n},gP)\leq \max\bigl\{\epsilon(2A+1)+\epsilon+A, A_1\bigr\}.\] It implies that $c\subset N_R(gP)$ for some $R$ by the uniform neighborhood quasiconvexity of $gP$. This is a contradiction. Thus, $\set{v_m}{m\geq 0}$ is infinite. Therefore, there is a geodesic ray $\hat{c}$ in $\hat{\Gamma}(G,S,\PP)$ emanating from $z_0$ and $d_m\subset\hat{c}$ for all $m$.

Each $G$--vertex $u$ of $\hat{c}$ lies in some segment $d_m$ and we know that $d_m\subset c_n$ when $n>m$, then $u$ lies in the $A$--neighborhood of $c$. We claim that $c$ lies in the $B$--neighborhood of $\Sat_0(\hat{c})$ for some $B(\epsilon)$ with respect to the metric $d_S$.

For each $G$--vertex $v$ of $\hat{c}$, let \[c_v=\Hull_c\bigl(c\cap B(v,A)\bigr).\] For each edge $e$ of $\hat{c}$ with $G$--endpoints $v$ and $w$, let \[c_e= \Hull_c(c_v\cup c_w).\] We note that $e$ is either an $S$--edge in $\Gamma(G,S)$ or $e$ is a peripheral edge that consists of two peripheral half edges with some $v_{gP}$ in the middle. Again, $c$ is covered by the sets $c_e$ for all edges $e$ of $\hat{c}$. Thus, for each $x\in c$, $x\in c_e$ for some edge $e$ of $\hat{c}$.

If $e$ is an $S$--edge, the length $c_e$ is at most $\epsilon(2A+1)+\epsilon$. It implies that $x$ lies within a distance $\epsilon(2A+1)+\epsilon$ of $c\cap B(v,A)$, where $v$ is a $G$--vertex incident to $e$. Thus, $x$ lies within $\epsilon(2A+1)+\epsilon+A$ of $v$. Therefore, $x$ lies within $\epsilon(2A+1)+\epsilon+A$ of $\Sat_0(\hat{c})$.
 
If $e$ is a peripheral edge, then the midpoint of $e$ is $v_{gP}$ for some $gP$ and $c\cap B(v,A)$ lies in the $A$--neighborhood of $gP$ for each $G$--endpoint $v$ of $e$. Since each point in $c_e$ lies between two such points, $c_e$ lies in the $A_1$--neighborhood of $gP$. In particular, $d_S(x,gP)\leq A_1$. Therefore, $c$ lies in the $B$--neighborhood of the $\Sat_0(\hat{c})$ with respect to $d_S$ for some $B(\epsilon)$.
\end{proof}

\begin{lem} \label{l15}
For each $\epsilon>0$ there are constants $A=A(\epsilon)>0$ and $B=B(\epsilon)>0$ such that the following holds. Let $c$ be an $\epsilon$--quasigeodesic ray in $\Gamma(G,S)$ such that $c$ is contained in $N_R(g^*P^*)$ for some peripheral left coset $g^*P^*$ and some positive number $R$. Let $\hat{c}$ be a geodesic in $\hat{\Gamma}(G,S,\PP)$ that connects the initial point of $c$ and $v_{g^*P^*}$. Then each $G$--vertex of $\hat{c}$ lies in the $A$--neighborhood of $c$, and $c$ lies in the $B$--neighborhood of $\Sat_0(\hat{c})$ with respect to the metric $d_S$. Moreover, if $z^*$ is a point in $c$ such that $d_S(z^*,g^*P^*)\leq A$, then the subray of $c$ that emanates from $z^*$ lies in the $B$--neighborhood of $g^*P^*$ with respect to the metric $d_S$.
\end{lem}

\begin{proof}
Let $A=A(\epsilon)$ be the constant from Lemma \ref{l3}. Since $(gP)_{{gP}\in \Pi}$ is a uniform neighborhood quasiconvex collection of subsets of $\Gamma(G,S)$, then there is some $A_1=A_1(A,\epsilon)\geq 0$ such that any $\epsilon$--quasigeodesic segment whose endpoints lie in the $A$--neighborhood of any set $gP$ must lie in the $A_1$--neighborhood of $gP$. Choose $(z_n)$ in c such that $z_0$ is the initial endpoint of $c$ and $d_S(z_0,z_n)\rightarrow \infty$. For each $n$, choose a geodesic $\hat{c}_n$ in $\hat{\Gamma}(G,S,\PP)$ with the endpoints $z_0$ and $z_n$.

By using the same argument as we used in Lemma \ref{l5}, we could assume that for each $m\geq 0$, there is a vertex $v_m$ in $\hat{\Gamma}(G,S,\PP)$ such that $v_m\in\hat{c}_n$ for all $n>m$. Moreover, if $v_m$ is $G$--vertex then $d_S(v_m, z_m) \leq\epsilon(2A+1)+\epsilon+A$ and if $v_m=v_{gP}$ for some peripheral left coset $gP$, then $d_S(z_m,gP)\leq A_1$.

Since $c$ lies in the $R$--neighborhood of $g^*P^*$, then the length of each $\hat{c}_n$ is at most $2R+1$. This implies that the set $\set{v_m}{m\geq 0}$ is finite. For each $m\geq 0$, if $v_m$ is a $G$--vertex, then we choose some $g_mP_m$ containing $v_m$, otherwise we choose $g_mP_m=gP$ where $v_m=v_{gP}$. Thus, the set $\set{g_mP_m}{m\geq 0}$ is also finite. Also $d_S(z_m,g_mP_m)\leq A_2$, where $A_2=\max\bigl\{\epsilon(2A+1)+\epsilon+A, A_1\bigr\}$.

We could find a subsequence $(m_n)$ and $g'P'\in \set{g_mP_m}{m\geq 0}$ such that $g_{m_n}P_{m_n}=g'P'$. This implies that $d_S(z_{m_n},g'P')\leq A_2$. Since $g'P'$ is a uniform neighborhood quasiconvex set, then there is a number $A_3$ such that $N_{A_3}(g'P')$ contains an infinite subray of $c$. In particular, $\diam\bigl(N_{A_3}(g'P')\cap N_R(g^*P^*)\bigr)$ is infinite. This implies that $g'P'=g^*P^*$ by the bounded penetration property.

If $v_{g^*P^*}$ does not lie in any $\hat{c}_n$, then $v_{m_n}$ must lie in $g_{m_n}P_{m_n}=g^*P^*$ and $d_S(z_{m_n},v_{m_n})\leq \epsilon(2A+1)+\epsilon+A$ for all $n$. It is a contradiction since the geodesic $\hat{c}_k$ contain more than two points in $g^*P^*$ for some $k$. Thus, $v_{g^*P^*}$ must lie in some $\hat{c}_n$.

We replace the geodesic segment of $\hat{c}_n$ that emanates from the initial point of $c$ to the point $v_{g^*P^*}$ by $\hat{c}$, then the new path $\hat{c}'_n$ is also a geodesic with two endpoints in $c$. This implies that each $G$--vertex of $\hat{c}'_n$ lies in the $A$--neighborhood of $c$. In particular, each $G$--vertex of $\hat{c}$ lies in the $A$--neighborhood of $c$.

If $z^*$ is a point in $c$ such that $d(z^*,g^*P^*)< A$, then the subray of $c$ that emanates from $z^*$ lies in the $A_3$--neighborhood of $g^*P^*$ for some $A_3$ since $d_S(z_{m_n},g^*P^*)\leq A_2$ and $g^*P^*$ is a uniform neighborhood quasiconvex set. Therefore, it is obvious that the subray of $c$ that emanates from $z^*$ lies in the $B$--neighborhood of $g^*P^*$ and $c$ lies in the $B$--neighborhood of $\Sat_0(\hat{c})$ for some $B(\epsilon)$ with respect to the metric $d_S$.
\end{proof}

\section{The connection among the $\CAT(0)$ geometry, the Cayley graph and the coned off Cayley graph}
\label{connection}

In this section, we build the connection between the space that a relatively hyperbolic group acts on and the associated graphs of the group (the Cayley graph and the coned off Cayley graph). This connection is the key step to prove the Main Theorem in Section \ref{Main Theorem}. 

From now, we denote the metric in $\Gamma(G,S)$ by $d_S$, the metric in $\hat{\Gamma}(G,S,\PP)$ by $d$, and the metric in $X$ by $d_X$. We suppose a finitely generated relatively hyperbolic group $G$ acts properly and cocompactly on a $\CAT(0)$ space $X$. Therefore, we have a $G$--equivariant map $\Phi\colon \Gamma(G,S)\to X$ that satisfies the inequality
\begin{equation} \label{eq:qi} \frac{1}{K}d_S(u,u')-1\leq d_X\bigl(\Phi(u),\Phi(u')\bigr)\leq Kd_S(u,u')\end{equation}
for all $u,u'\in \Gamma(G,S)$ and $N_K\bigl(\Phi(\Gamma(G,S))\bigr)=X$ for some $K\geq1$. In particular, $\Phi$ is a $K$--quasi-isometry. We note the reader that we will use the inequality \eqref{eq:qi} many times in the rest of the paper.

%\begin{lem}
%\label{quasi}
%For each $\epsilon>0$ there is a positive number $L=L(\epsilon)$ such that if $c$ is an $\epsilon$--quasigeodesic in $\Gamma(G,S)$, then $\Phi(c)$ is an $L$--quasigeodesic in $X$.
%\end{lem} 
%
%The proof of this lemma is elementary and we leave it to the reader.

We define $Y_{gP}=\Phi(gP)$ for each peripheral left coset $gP$ and we call it a \emph{peripheral space}.

For each peripheral space $Y_{gP}$, we define its boundary in $X$, denoted $\partial Y_{gP}$, to be the set of all $[\gamma]\in\partial X$ where $\gamma \subset N_R(Y_{gP})$ for some R. Each element in $\partial Y_{gP}$ is said to be a \emph{peripheral limit point}. 

We observe that each $Y_{gP}$ depends on the $G$--equivariant quasi-isometric map $\Phi$ but $\partial Y_{gP}$ does not. Moreover, the group $G$ acts on the $\CAT(0)$ boundary of $X$ and the set of all peripheral limit points is invariant under the action of $G$. 

\begin{lem} 
$(Y_{gP})_{{gP}\in \Pi}$ is a quasidense, locally finite, bounded penetration and uniform neighborhood quasiconvex collection of subsets in $X$.
\end{lem}

\begin{proof}
It is obvious since $\Gamma(G,S)$ and $X$ are quasi-isometric under $\Phi$, the collection $(gP)_{{gP}\in \Pi}$ is quasidense, locally finite, bounded penetration and uniformly neighborhood quasiconvex in $\Gamma(G,S)$, and $(Y_{gP})_{{gP}\in \Pi}$ is the image of $(gP)_{{gP}\in \Pi}$ under the quasi-isometric map $\Phi$.
\end{proof}

\begin{lem} \label{l31}
$\partial Y_{gP}$ is non--empty iff $gP$ is infinite and $\partial Y_{gP}\cap \partial Y_{g'P'}=\emptyset$ when $gP\neq g'P'$.
\end{lem}

\begin{proof}
The first statement is obvious and the second statement is implied by the bounded penetration property of $(Y_{gP})_{{gP}\in \Pi}$.
\end{proof}

From the result of Lemma \ref{l31}, the following concept is an equivalent relation:
\begin{defn}
Two peripheral limit points are said to be of \emph{the same type} if they both lie in $\partial Y_{gP}$ for some peripheral left coset $gP$.
\end{defn} 

\begin{lem} \label{l32}
If $x$ and $y$ are two peripheral limit points of the same type and $g$ is any group element in $G$, then $gx$ and $gy$ are also two peripheral limit points of the same type.
\end{lem}

The proof for this lemma is obvious and we leave it to the reader. 

\begin{rem}
From the result of Lemma \ref{l32}, we see that the group $G$ acts on the space obtained from $\partial X$ by identifying all peripheral limit points of the same type.
\end{rem}

\begin{lem} \label{l4}
There are constants $\epsilon >0$, $r>0$ such that the following holds. If $\alpha$ is a geodesic in $X$, then there is an $\epsilon$--quasigeodesic c in $\Gamma(G,S)$ such that the Hausdorff distance between $\Phi(c)$ and $\alpha$ is at most $r$. Moreover, if $\alpha$ is a geodesic segment with two endpoints $\Phi(g)$ and $\Phi(h)$, where $g,h\in G$, then $c$ could be chosen with endpoints $g$ and $h$. If $\alpha$ is a geodesic ray with initial point $\Phi(g)$, where $g\in G$, then $c$ could be chosen with initial points $g$.
\end{lem}

\begin{proof} 
It is obvious since $X$ and $\Gamma(G,S)$ are quasi-isometric under the map $\Phi$.
\end{proof}

\begin{defn}
A pair of paths $(c,\hat{c})$ in $\hat{\Gamma}(G,S,\PP)$ is said to be an \emph{$(\epsilon, A)$--nice pair} if $c$ is an $\epsilon$--quasigeodesic ray in $\Gamma(G,S)$, $\hat{c}$ is a geodesic ray in $\hat{\Gamma}(G,S,\PP)$ and each $G$--vertex of $\hat{c}$ lies in the $A$--neighborhood of $c$ with respect to the metric $d_S$. A pair of paths $(c,\hat{c})$ in $\hat{\Gamma}(G,S,\PP)$ is said to be a \emph{nice pair} if it is an $(\epsilon, A)$--nice pair for some $\epsilon, A$. A pair of paths $(c,\hat{c})$ in $\hat{\Gamma}(G,S,\PP)$ is said to be an \emph{$(\epsilon, A, r)$--nice pair of some geodesic $\alpha$ in $X$} if $(c,\hat{c})$ is an $(\epsilon, A)$--nice pair and the Hausdorff distance between $\Phi(c)$ and $\alpha$ is at most $r$. A pair of paths $(c,\hat{c})$ in $\hat{\Gamma}(G,S,\PP)$ is said to be a \emph{nice pair of some geodesic $\alpha$ in $X$} if $(c,\hat{c})$ is an $(\epsilon, A, r)$--nice pair of $\alpha$ for some $\epsilon, A, r$.
\end{defn}

\begin{rem}
If $(c,\hat{c})$ is an nice pair in $\hat{\Gamma}(G,S,\PP)$, the Hausdorff distance between them is finite with respect to the metric $d$.
\end{rem}

\begin{lem}\label{l7}
There are positive constants $\epsilon, A, r, B$ such that the following holds. Let $\alpha$ be a geodesic ray in $X$ such that $[\alpha]\notin \partial Y_{gP}$ for any $gP$. Then there is an $(\epsilon, A, r$)--nice pair $(c,\hat{c})$ of $\alpha$ such that $c$ lies in the $B$--neighborhood of $\Sat_0(\hat{c})$ with respect to $d_S$. Moreover, if $\alpha$ is a geodesic ray with initial point $\Phi(g)$, where $g\in G$, then $c$ and $\hat{c}$ could be chosen with initial points $g$. 
\end{lem}

\begin{proof} 
The lemma can be proved by using Lemma \ref{l4}, Lemma \ref{l5}, and the facts $X$ and $\Gamma(G,S)$ are quasi-isometric under $\Phi$ and $\bigl(Y_{gP}\bigr)_{{gP}\in \Pi}$ is the image of $(gP)_{{gP}\in \Pi}$ under the quasi-isometric map $\Phi$.
\end{proof}

The following two lemmas will allow us to define a bijection from the set of non-peripheral limit points of $\partial X$ to the Gromov boundary of the coned off Cayley graph $\hat{\Gamma}(G,S,\PP)$. Lemma \ref{l9} shows this map is surjective and Lemma \ref{l11} shows it is injective.

\begin{lem} \label{l9}
For each geodesic ray $\pi$ in $\hat{\Gamma}(G,S,\PP)$ there is a geodesic ray $\gamma$ in $X$ such that $[\gamma]\notin \partial Y_{gP}$ for any peripheral left coset $gP$ and the following holds. If $(c,\hat{c})$ is an arbitrary nice pair of $\gamma$, then the Hausdorff distance between $\pi$ and $\hat{c}$ is finite in $\hat{\Gamma}(G,S,\PP)$, or $\pi$ and $\hat{c}$ are equivalent geodesic rays in $\hat{\Gamma}(G,S,\PP)$.
\end{lem}

\begin{proof}
Let $\left\{g_n\right\}$ be the set of all $G$--vertices of $\pi$.
For each $n\geq 1$ choose a geodesic $\gamma_n$ connecting $\Phi(g_0)$ and $\Phi(g_n)$ in X.
By Lemma \ref{l4}, there is an $\epsilon$--quasigeodesic $c_n$ with the endpoints $g_0$ and $g_n$ such that the Hausdorff distance between $\Phi(c_n)$ and $\gamma_n$ is at most $r$ for some uniform constants $\epsilon$ and $r$. 
There is a subsequence $(\gamma_{m_n}$) of $(\gamma_n$) that approaches to a geodesic $\gamma$ in $X$.

For each $m\geq 0$ and $m_n>m$, there is a point $z_{m_n}$ in $c_{m_n}$ such that $d_S(z_{m_n}, g_m)\leq A$ for some $A=A(\epsilon)$ defined in Lemma \ref{l3}. By inequality \eqref{eq:qi}, $d_X\bigl(\Phi(g_m),\Phi(z_{m_n})\bigr)\leq KA$. Since the Hausdorff distance between $\Phi(c_{m_n})$ and $\gamma_{m_n}$ is at most $r$, then there is $w_{m_n}\in \gamma_{m_n}$ such that $d_X\bigl(\Phi(z_{m_n}), w_{m_n}\bigr)\leq r$. This implies that $d_X\bigl(\Phi(g_m), w_{m_n}\bigr)\leq KA+r$. We could assume that $w_{m_n}$ approaches $w\in \gamma$. Thus, $\Phi(g_m)$ lies in the $(KA+r)$--neighborhood of $\gamma$.

For each $n\geq1$, choose $\alpha_n$ to be a geodesic that connects $\Phi(g_{n-1})$ and $\Phi(g_{n})$ and we define $\alpha=\alpha_1\cup\alpha_2\cup\alpha_3\cdots$. Since the endpoints of each $\alpha_i$ lie in the $(KA+r)$--neighborhood of $\gamma$, then $\alpha$ also lies in the $(KA+r)$--neighborhood of $\gamma$ by Lemma \ref{l6}. Since $d_X\bigl(\Phi(g_0), \Phi(g_n)\bigr)$ approaches infinity, then $\alpha$ is an infinite ray. Thus, $\gamma$ lies in the $M_1$--neighborhood of $\alpha$ for some $M_1$.

If $\gamma\subset N_R(Y_{gP})$ for some positive number $R$ and some peripheral left coset $gP$, then $\alpha \subset N_{R+KA+r}(Y_{gP})$. In particular, $\Phi(g_n)$ lies in the $N_{R+KA+r}(Y_{gP})$ for all $n$. This implies $g_n$ lies in the $\bigl(K(R+KA+r)+K\bigr)$--neighborhood of $gP$ for all $n$ by the inequality \eqref{eq:qi}. Thus, $d(g_0,g_n)$ are bounded. This is a contradiction since $\pi$ is a geodesic. Thus, $[\gamma]\notin\partial Y_{gP}$ for any peripheral left coset $gP$.

Let $(c,\hat{c})$ be any nice pair of $\gamma$. Thus, the Hausdorff distance between $\gamma$ and $\Phi(c)$ is at most $r'<\infty$ and $G$--vertices of $\hat{c}$ lie in some $A'$--neighborhood of $c$ with respect to $d_S$. (The existence of $(c,\hat{c})$ is guaranteed by Lemma \ref{l7}.) 

For each $u$ in $\hat{c}$ there is $v$ in $c$ such that $d_S(u,v)\leq A'$. Since the Hausdorff distance between $\Phi(c)$ and $\gamma$ is at most $r'$, then there is $z$ in $\gamma$ such that $d_X\bigl(\Phi(v),z\bigr)\leq r'$. Since $\gamma$ lies in the $M_1$--neighborhood of $\alpha$, there is $w$ in some $\alpha_n=\bigl[\Phi(g_{n-1}), \Phi(g_n)\bigr]$ such that $d_X(z,w)\leq M_1$. Thus, $d_X\bigl(\Phi(v),w\bigr)\leq M_1+r'$.

If $g_{n-1}$ and $g_n$ lie in some peripheral left coset $gP$, then $\Phi(g_{n-1})$ and $\Phi(g_{n})$ lie in $Y_{gP}$. This implies $w$ lies in some $M_2$--neighborhood of $Y_{gP}$ by the uniform neighborhood quasiconvexity of $Y_{gP}$. Thus, there is $g'$ in $gP$ such that $d_X\bigl(w,\Phi(g')\bigr)\leq M_2$, then $d_X\bigl(\Phi(v),\Phi(g')\bigr)\leq M_1+r'+M_2$. Thus, $d_S(v,g')\leq K(M_1+r'+M_2)+K$ by the inequality \eqref{eq:qi}. This implies that \[d_S(u,g')\leq d_S(v,g')+ d_S(u,v) \leq K(M_1+r'+M_2)+K+A'.\] Also if $g_n$ and $g'$ lie in the same peripheral left coset $gP$, then $d(g_n,g')\leq 1$. Therefore, \[d(u,g_n)\leq d(u,g')+d(g',g_n)\leq d_S(u,g')+d(g',g_n)\leq K(M_1+r'+M_2)+K+A'+1.\]

If there is no peripheral left coset that contains both $g_{n-1}$ and $g_n$, then $g_{n-1}$ and $g_n$ are two vertices of an $S$--edge of $\pi$. Since $d_S(g_{n-1},g_n)=1$, then $d_X\bigl(\Phi(g_{n-1}),\Phi(g_n)\bigr)\leq K$. This implies $d_X\bigl(w,\Phi(g_n)\bigr)\leq K$. Thus, \[d_X\bigl(\Phi(v),\Phi(g_n)\bigr)\leq d_X\bigl(w,\Phi(g_n)\bigr)+d_X\bigl(\Phi(v),w\bigr)\leq K+r'+M_1.\] This implies that $d_S(v,g_n)\leq K(K+r'+M_1)+K$ by the inequality \eqref{eq:qi}. Thus, \[d_S(u,g_n)\leq d_S(v,g_n)+d_S(v,u)\leq K(K+r'+M_1)+K+A'.\] This implies that \[d(u,g_n)\leq K(K+r'+M_1)+K+A'.\] Therefore, $\hat{c}$ lies in a bounded neighborhood of $\pi$ with respect to $d$.
Since $\hat{c}$ is a geodesic ray, then $\pi$ also lies in a bounded neighborhood of $\hat{c}$.
Therefore, in both cases, the Hausdorff distance between $\pi$ and $\hat{c}$ with respect to $d$ is finite, or $\pi$ and $\hat{c}$ are equivalent geodesic rays in $\hat{\Gamma}(G,S,\PP)$.
\end{proof}

\begin{lem} \label{l11}
Let $\alpha$ and $\alpha'$ be two geodesic rays in $X$ with the same initial point $\Phi(h_0)$, where $h_0\in G$. Let $(c,\hat{c})$ and $(c',\hat{c}')$ in $\hat{\Gamma}(G,S,\PP)$ be $(\epsilon, A, r)$--nice pairs of $\alpha$ and $\alpha'$ respectively. Suppose that the initial points of all $c,\hat{c},c',\hat{c}'$ are $h_0$ and $\hat{c}$, $\hat{c}'$ are two equivalent geodesic rays in $\hat{\Gamma}(G,S,\PP)$. Then $\alpha,\alpha'$ are two equivalent geodesic rays in $X$.
\end{lem}

\begin{proof}
Let $(g_n)$ and $(g'_n)$ be the sequences of all $G$--vertices of $\hat{c}$ and $\hat{c}'$ respectively. Let $(x_n)$ and $(x'_n)$ be the sequences of vertices of $c$ and $c'$ respectively such that $d_S(g_n,x_n)\leq A$ and $d_S(g'_n,x'_n)\leq A$.

By Lemma \ref{l8}, there is a constant $v$ such that the Hausdorff distance between $(g_n)$ and $(g'_n)$ is at most $v$ with respect to the metric $d_S$. This implies that the Hausdorff distance between $(x_n)$ and $(x'_n)$ is at most $v+2A$ with respect to the metric $d_S$. Thus, the Hausdorff distance between $\Phi(x_n)$ and $\Phi(x'_n)$ is at most $K(v+2A)$ by the inequality \eqref{eq:qi}.

Let $(w_n)$ and $(w'_n)$ be the sequences of vertices of $\alpha$ and $\alpha'$ respectively such that $d_X\bigl(w_n,\Phi(x_n)\bigr)\leq r$ and $d_X\bigl(w'_n,\Phi(x'_n)\bigr)\leq r$. This implies that the Hausdorff distance between $(w_n)$ and $(w'_n)$ is at most $K(v+2A)+2r$. 
Also \begin{align*} d_X\bigl(\Phi(h_0),w_n\bigr)&\geq d_X\bigl(\Phi(h_0),\Phi(x_n)\bigr)-d_X\bigl(w_n,\Phi(x_n)\bigr)\\&\geq \frac {1} {K} d_S(h_0,x_n)-1-r \\&\geq \frac {1} {K}\bigl(d_S(h_0,g_n)-d_S(g_n,x_n)\bigr)-1-r\\&\geq \frac {1}{K}\bigl(d(h_0,g_n)-A\bigr)-1-r \rightarrow \infty \end{align*}
Similarly, $d_X\bigl(\Phi(h'_0),w'_n\bigr) \rightarrow \infty$.

For each $x$ in $\alpha$, $x$ must lie in [$w_{n-1},w_n$] for some $n$. Since $w_{n-1},w_n$ lie in the $\bigl(K(v+2A)+2r\bigr)$--neighborhood of $\alpha'$, then $x$ also lies in the $\bigl(K(v+2A)+2r\bigr)$--neighborhood of $\alpha'$ by Lemma \ref{l6}. Thus, $\alpha$ lies in the $\bigl(K(v+2A)+2r\bigr)$--neighborhood of $\alpha'$. Similarly, $\alpha'$ lies in the $\bigl(K(v+2A)+2r\bigr)$--neighborhood of $\alpha$.
Therefore, $\alpha$, $\alpha'$ are two equivalent rays in $X$.
\end{proof}

\begin{lem} \label{l13}
Let $\alpha$ be a geodesic ray in $X$ such that $[\alpha]\notin \partial Y_{gP}$ for any $gP$ and $(c,\hat{c})$ be an $(\epsilon,A,r)$--nice pair of $\alpha$. Then, for each $g\in G$, $(gc,g\hat{c})$ is an $(\epsilon,A,r)$--nice pair of $g\alpha$.
\end{lem}

\begin{proof}
It is obvious since $G$ acts isometrically on $X$, $\hat{\Gamma}(G,S,\PP)$, $\Gamma(G,S)$ and the map $\Phi$ is $G$--equivariant.
\end{proof}

The proof for the following lemma is obvious and we leave it to the reader. 

\begin{lem}\label{l101}
Let $\alpha$ be a geodesic ray in $X$ such that $[\alpha]\in \partial Y_{gP}$ for some $gP$ and h be any group element in $G$. Then $h[\alpha]\in \partial Y_{hgP}$.
\end{lem}

\section{Main Theorem} \label{Main Theorem}

In this section, we will show the connection between the $\CAT(0)$ boundary of $X$ and the relative boundary of a relatively hyperbolic group $G$ that acts on X properly and cocompactly.

Now, we build the map $f\colon \partial X\to \Delta_{\infty}\hat{\Gamma}$ between the $\CAT(0)$ boundary of $X$ and the infinite hyperbolic closure $\Delta_{\infty}\hat{\Gamma}$ of $\hat{\Gamma}(G,S,\PP)$ as follows:

Let $[\alpha]$ be a point in $\partial X$. If $[\alpha]\in \bigcup_{gP\in\Pi} \partial Y_{gP}$, then there is a unique peripheral left coset $g_0P_0$ such that $[\alpha]\in \partial Y_{g_0P_0}$. We define $f\bigl([\alpha]\bigr)=v_{g_0P_0}$. If $[\alpha]\notin \bigcup_{gP\in\Pi} \partial Y_{gP}$, then there is a nice pair $(c,\hat{c})$ in $\hat{\Gamma}(G,S,\PP)$ such that the Hausdorff distance between $\Phi(c)$ and $\alpha$ is finite. We define $f\bigl([\alpha]\bigr)=[\hat{c}]$.

\begin{lem}
The map $f$ is well-defined.
\end{lem}

\begin{proof}
Suppose [$\alpha_1$]= [$\alpha_2$] in $\partial X$, where $\alpha_1$ and $\alpha_2$ are two geodesic rays in $X$.
If one of them belongs to $\bigcup_{gP\in\Pi} \partial Y_{gP}$\ then there is a unique peripheral left coset $g_0P_0$ such that $[\alpha_1] = [\alpha_2] \in \partial Y_{g_0P_0}$. Therefore, $f\bigl([\alpha_1]\bigr)= f\bigl([\alpha_2]\bigr) = v_{g_0P_0}$. Suppose that [$\alpha_1$]= [$\alpha_2$] lies in $\partial X-\bigcup_{gP\in\Pi}\partial Y_{gP}$. Let $(c_1,\hat{c}_1)$ and $(c_2,\hat{c}_2)$ be two nice pairs of $\alpha_1$ and $\alpha_2$ respectively. Thus, the Hausdorff distances $d_H\bigl(\Phi(c_1),\alpha_1\bigr)$ and $d_H\bigl(\Phi(c_2),\alpha_2\bigr)$ are finite. This implies that the Hausdorff distance $d_H\bigl(\Phi(c_1),\Phi(c_2)\bigr)<\infty$. Thus, the Hausdorff distance between $c_1$ and $c_2$ is finite with respect to the metric $d_S$. This implies that the Hausdorff distance between $c_1$ and $c_2$ is also finite with respect to the metric $d$. Since the Hausdorff distance between $c_i$ and $\hat{c}_i$ is finite with respect to the metric $d$ for $i=1,2$, then the Hausdorff distance between $\hat{c}_1$ and $\hat{c}_2$ with respect to the metric $d$ is finite i.e $[\hat{c}_1]=[\hat{c}_2]$.
\end{proof}

\begin{lem}
The map $f$ maps the set $\partial X-\bigcup_{gP\in\Pi}\partial Y_{gP}$ onto $\partial \hat{\Gamma}$.
\end{lem}

\begin{proof}
It is guaranteed by the definition of $f$ and Lemma \ref{l9}.
\end{proof}

\begin{lem}
The map $f$ maps the set $\bigcup_{gP\in\Pi}\partial Y_{gP}$ onto $V_{\infty}(\hat{\Gamma})$.
\end{lem}

\begin{proof}
We have $f\bigl(\bigcup_{gP\in\Pi}\partial Y_{gP}\bigr)\subseteq V_{\infty}(\hat{\Gamma})$ by construction.
For each $v_{gP}\in V_{\infty}(\hat{\Gamma})$, since $gP$ is infinite, we could choose $[\alpha]\in \partial Y_{gP}$ and we have $f\bigl([\alpha]\bigr) = v_{gP}$.
Therefore, $f\bigl(\bigcup_{gP\in\Pi}\partial Y_{gP}\bigr)\supset V_{\infty}(\hat{\Gamma})$.
\end{proof}

\begin{lem}
The map $f$ is injective in $\partial X-\bigcup_{gP\in\Pi}\partial Y_{gP}$.
\end{lem}

\begin{proof}
It is guaranteed by Lemma \ref{l11}.
\end{proof}

\begin{lem}
The map $f$ is $G$--equivariant. 
\end{lem}

\begin{proof}
It is guaranteed by Lemma \ref{l13}, Lemma \ref{l101} and the fact that $hv_{gP}=v_{hgP}$ for any peripheral vertex $v_{gP}$ and any group element $h\in G$.
\end{proof}

\begin{rem}
We are now going to show the continuity of the function $f$. This is the most technical part in this paper. Our strategy is to prove first that $f$ is continuous at each non-peripheral limit point and later at each peripheral limit point. Before starting the proof, we need to fix some constants and some notation for convenience:

Let $\epsilon,r$ be the numbers in Lemma \ref{l7} and Lemma \ref{l4}. 

Let $A_1$ and $B_1$ be the numbers in Lemma \ref{l7}, let $A_2$ and $B_2$ be the numbers in Lemma \ref{l15} (We note that the numbers $A_2$ and $B_2$ in Lemma \ref{l15} depend on the number $\epsilon$ we have just chosen above.) Let $A=\max\{A_1,A_2\}$ and $B=\max\{B_1,B_2\}$.  

Let $K$ be the constant in the inequality \eqref{eq:qi}.

For any $x\in X$ and $\xi \in \partial X$ we denote $[x,\xi)$ to be the unique geodesic from $x$ to $\xi$. 
\end{rem}

We are now ready to prove the continuity of $f$ at each non-peripheral limit point.

\begin{prop} \label{loc1}
The map $f$ is continuous at each non-peripheral limit point.
\end{prop}

\begin{proof}
Let $\xi$ be an arbitrary point in $\partial X-\bigcup_{gP\in\Pi}\partial Y_{gP}$. Thus, $\eta=f(\xi)\in \partial \hat{\Gamma}$. Let $W$ be a neighborhood of $\eta$ in $\Delta_{\infty}\hat{\Gamma}$. We will show that there is a neighborhood $U$ of $\xi$ such that $f(U)\subset W$. Since $\Delta_{\infty}\hat{\Gamma}$ is a subspace of $\Delta\hat{\Gamma}$, then there is a finite subset $D$ of $V(\hat{\Gamma})$ such that $M(\eta,D)\cap\Delta_{\infty}\hat{\Gamma}\subset W$. Thus, it is sufficient to find a neighborhood $U$ of $\xi$ such that $f(U)\subset M(\eta,D)\cap\Delta_{\infty}\hat{\Gamma}$.

The idea here is to find $U$ of the form $U=U\bigl(\Phi(h_0),\xi,t,1\bigr)$ (see Definition \ref{CAT(0)}), where $h_0$ is some $G$--vertex and t is a large enough number such that $f(U)\subset M(\eta,D)\cap\Delta_{\infty}\hat{\Gamma}$. We construct $U$ as follows:

Let $\alpha= [\Phi(h_0),\xi)$ for some $G$--vertex $h_0$. Let $\sigma =\max\set{d(h_0,a)}{a\in D}$. Let $R=R(\delta,\sigma)$ be the constant in Lemma \ref{l14}, where $\delta$ is the hyperbolic constant of $\hat{\Gamma}(G,S,\PP)$.
Let \begin{align*} n_1&=K(KA+KB+2r+1)+K+1 \\ n_0&= \sigma+2R+n_1+1\end{align*} Let $(c,\hat{c}$) be an $(\epsilon, A,r)$--nice pair of $\alpha$ such that $c$ and $\hat{c}$ have the same initial point $h_0$. Thus, $f(\xi)=f\bigl([\alpha]\bigr)=[\hat{c}]$. Let $h^*$ be a $G$--vertex in $\hat{c}$ such that $d(h_0,h^*)\geq n_0$. Choose $y^*$ in $c$ such that $d_S(h^*,y^*)\leq A$. Choose a positive number $t$ such that $d_X\bigl(\Phi(y^*),\alpha(t)\bigr)\leq r$ and define $U=U\bigl(\Phi(h_0),\xi,t,1\bigr)$. 

We first observe that \begin{align*} d_X\bigl(\Phi(h^*),\alpha(t)\bigr)&\leq d_X\bigl(\Phi(h^*),\Phi(y^*)\bigr)+d_X\bigl(\Phi(y^*),\alpha(t)\bigr)\\&\leq Kd_S(h^*,y^*)+d_X\bigl(\Phi(y^*),\alpha(t)\bigr)\\&\leq KA+r.\end{align*}

We now try to prove $f(U)\subset M(\eta,D)\cap\Delta_{\infty}\hat{\Gamma}$. For each $\xi'$ in $U$, we denote $\alpha'= [\Phi(h_0),\xi')$. Thus, $d_X\bigl(\alpha'(t),\alpha(t)\bigr)\leq 1$. Let $\pi$ be a geodesic in $\hat{\Gamma}$ connecting $\hat{c}$ and $f(\xi')$. In order to show $f(\xi')$ is an element in $M(\eta,D)$, we need to prove $\pi \cap D =\emptyset$. In order to see this fact, we need to consider two cases: $\xi'$ is a non-peripheral limit point and $\xi'$ is a peripheral limit point. The following two lemmas will help us finish the proof of this proposition. 
\end{proof}

\begin{lem} \label{loc1a}
If $\xi'$ is a non-peripheral limit point, then $\pi \cap D =\emptyset$. 
\end{lem}
 
\begin{proof}
Let $(c',\hat{c}'$) be an $(\epsilon, A,r)$--nice pair of $\alpha'$ such that $c'$ and $\hat{c}'$ have the same initial point $h_0$. Thus, $f(\xi')=f\bigl([\alpha']\bigr)=[\hat{c}']$ and $\pi$ is a geodesic in $\hat{\Gamma}$ connecting $\hat{c}$ and $\hat{c}'$. 

Since $c'$ lies in the $B$--neighborhood of $\Sat_0(\hat{c}')$, then $\Phi(c')$ lies in the $KB$--neighborhood of $\Phi(\Sat_0(\hat{c}'))$. Also the Hausdorff distance between $\Phi(c')$ and $\alpha'$ is at most $r$. Thus, $\alpha'(t)$ lies in the $(KB+r)$--neighborhood of $\Phi(\Sat_0(\hat{c}'))$. Therefore, there is a $G$--vertex $g_1$ of $\hat{c}'$ such that $\alpha'(t)$ lies in the $(KB+r)$--neighborhood of some peripheral left space $Y_{g_1P_1}$. We claim that the distance between $g_1$ and $h^*$ in $\hat{\Gamma}(G,S,\PP)$ is bounded by $n_1$. Indeed \begin{align*} d_X\bigl(Y_{g_1P_1},\Phi(h^*)\bigr)&\leq d_X\bigl(Y_{g_1P_1},\alpha'(t)\bigr)+ d_X\bigl(\alpha'(t),\alpha(t)\bigr)+d_X\bigl(\alpha(t),\Phi(h^*)\bigr)\\&\leq KB+r+1+KA+r\end{align*} and $Y_{g_1P_1}=\Phi(g_1P_1)$.\\Thus, \[d_S(g_1P_1, h^*)\leq K(2r+1+KA+KB)+K.\] Also, the $\diam(g_1P_1)\leq 1$ in the metric $d$.\\Therefore, \[d(g_1,h^*)\leq K(2r+1+KA+KB)+K+1\leq n_1.\] 

Assume that $\pi \cap D \neq \emptyset$. We will show that $n_0\leq \sigma+2R+n_1$ and it leads to a contradiction for the choice of $n_0$. Since $\sigma =\max\set{d(h_0,a)}{a\in D}$, then we could choose $z$ in $\pi$ such that $d(h_0,z)=d(h_0,\pi)\leq \sigma$. We could consider $\pi$ as the union of two rays $\pi^+$, $\pi^-$ with the same initial point $z$ such that $[\pi^+]=\xi$ and $[\pi^-]=\xi'$. Choose $z_1$ in $\pi^+$ and $z'_1$ in $\pi^-$ such that $d(h^*, z_1)\leq R$ and $d(g_1, z'_1)\leq R$. Obviously, $z$ lies between $z_1$ and $z'_1$ (i.e., $d(z'_1,z_1)=d(z'_1,z)+d(z,z_1)$). This implies that \begin{align*} d(g_1,h^*)&\geq d(z'_1, z_1)-d(g_1,z'_1)-d(h^*, z_1)\\&\geq d(z'_1,z)+d(z,z_1)-2R\\&\geq \bigl(d(h_0, g_1)-d(h_0, z)-d(g_1, z'_1)\bigr)+\bigl(d(h_0, h^*)-d(h_0,z)-d(h^*,z_1)\bigr)-2R\\&\geq \bigl(d(h_0, g_1)-\sigma-R\bigr)+\bigl(d(h_0, h^*)-\sigma-R\bigr)-2R\\&\geq \bigl(d(h_0, h^*)-d(g_1,h^*)\bigr)+d(h_0, h^*)-2\sigma-4R\\&\geq 2n_0-2\sigma-4R-n_1\end{align*} Also $d(g_1,h^*)\leq n_1$. This implies that $n_0\leq \sigma+2R+ n_1$. This contradicts the choice of $n_0$. Thus, $\pi \cap D =\emptyset$.
\end{proof} 

\begin{lem} \label{loc1b}
If $\xi'$ is a peripheral limit point, then $\pi \cap D =\emptyset$. 
\end{lem}

\begin{proof}
Suppose that $\xi'$ is an element of $\partial Y_{gP}$ for some peripheral left coset $gP$. Thus, $f(\xi')=v_{gP}$. Let $\hat{c}'$ be a geodesic in $\hat{\Gamma}$ connecting $h_0$ to $v_{gP}$. Choose $c'$ to be an $\epsilon$--quasigeodesic in $\Gamma(G,S)$ with the initial point $h_0$ such that the Hausdorff distance between $\Phi(c')$ and $\alpha'$ is at most $r$. Since $\alpha'$ lies in some $R_1$--neighborhood of $Y_{gP}$, then $\Phi(c')$ lies in the $(R_1+r)$--neighborhood of $Y_{gP}$. This implies that $c'$ lies in the $\bigl(K(R_1+r)+K\bigr)$--neighborhood of $gP$ by the inequality \eqref{eq:qi}. Thus, $c'$ lies in the $B$--neighborhood of $\Sat_0(\hat{c}')$. We use a similar argument as we used in Lemma \ref{loc1a} to have $n_0\leq \sigma+2R+ n_1$. This contradicts the choice of $n_0$. Thus, $\pi \cap D =\emptyset$.
\end{proof}

We now prove the continuity of $f$ at each peripheral limit point.

\begin{prop}
The map $f$ is continuous at each peripheral limit point.
\end{prop}
\begin{proof}

Let $\xi$ be an arbitrary peripheral limit point. Then $\xi$ lies in $\partial Y_{g_0P_0}$ for some peripheral left coset $g_0P_0$. Thus, $\eta=f(\xi)=v_{g_0P_0}$. Let $W$ be a neighborhood of $\eta$ in $\Delta_{\infty}\hat{\Gamma}$. Choose a finite subset $D$ of vertices of $\hat{\Gamma}$ that does not contain $v_{g_0P_0}$ such that $M_{(1,1)}(v_{g_0P_0},D)\cap\Delta_{\infty}\hat{\Gamma}\subset W$. As in Proposition \ref{loc1}, we must find a neighborhood $U=U\bigl(\Phi(h_0),\xi,t,1\bigr)$ of $\xi$ such that $f(U)\subset M_{(1,1)}(\eta,D)\cap\Delta_{\infty}\hat{\Gamma}$. In this case, we will choose $h_0$ to be a $G$--vertex in $g_0P_0$. We construct $U$ as follows:

Let $D_1$ be the set of all $G$--vertices of $D$ and $D_2$ be the set of all peripheral vertices of $D$. Let $Y=\Phi(D_1)\cup \bigcup_{v_{gP}\in D_2}Y_{gP}$ and $M_1=KA+r+1$. Since $\Phi(D_1)$ is finite, $\xi$ lies in $\partial Y_{g_0P_0}$ and $v_{g_0P_0}$ does not lie in $D_2$, then there is $h_0\in g_0P_0-\bigl(D_1\cup\bigcup_{v_{gP}\in D_2}gP\bigr)$ such that $\bigl[\Phi(h_0),\xi\bigr)\cap N_{M_1}(Y)= \emptyset$. We denote $\alpha=\bigl[\Phi(h_0),\xi\bigr)$. Since $\xi$ is a peripheral limit point in $\partial Y_{g_0P_0}$, then $\alpha$ lies in some $R_1$--neighborhood of $Y_{g_0P_0}$.

Let \begin{align*} \ell_1&=\max\set{d_X\bigl(\Phi(h_0),b\bigr)}{b\in\Phi(D_1)}\\ \ell_2&=\max\set{d_X\bigl(\Phi(h_0),Y_{gP}\bigr)}{v_{gP}\in D_2}\\\ell&=\max\{\ell_2, KA+r\}\end{align*}
$M_2=M_2(\ell)$ be the constant such that any geodesic segment whose endpoints lie in the $\ell$--neighborhood of any set $Y_{gP}$ must lie in the $M_2$--neighborhood of $Y_{gP}$. Let
$t=\max\set{\diam(N_{1+R_1}(Y_{g_0P_0})\cap N_{M_2+1}(Y_{gP}))}{v_{gP}\in D_2}+\ell_1+KA+r+1$ and $U=U\bigl(\Phi(h_0),\xi,t,1\bigr)$.

We now try to prove $f(U)\subset M(\eta,D)\cap\Delta_{\infty}\hat{\Gamma}$. For each $\xi'$ in $U$, we denote $\alpha'= [\Phi(h_0),\xi')$. Thus, $d_X\bigl(\alpha'(t),\alpha(t)\bigr)\leq 1$. In order to show $f(\xi')$ is an element in $M(\eta,D)$, we need to prove there is a $(1,1)$--quasigeodesic arc connecting $\eta=f(\xi)=v_{g_0P_0}$ and $f(\xi')$ that does not meet $D$. Our strategy is to find a geodesic $\hat{c}$ in $\hat{\Gamma}$ connecting $h_0$ and $f(\xi')$ first. If $\hat{c}$ contains $v_{g_0P_0}$, then we have a geodesic connecting $\eta=f(\xi)=v_{g_0P_0}$ and $f(\xi')$ that does not meet $D$. Otherwise, we extend $\hat{c}$ by attaching the peripheral half edge connecting $v_{g_0P_0}$ and the initial point $h_0$ of $\hat{c}$, then we have a $(1,1)$--quasigeodesic arc connecting $\eta=f(\xi)=v_{g_0P_0}$ and $f(\xi')$ that does not meet $D$. In order to see the existence of such a geodesic $\hat{c}$, we need to consider two cases: $\xi'$ is a non-peripheral limit point and $\xi'$ is a peripheral limit point. The following two lemmas complete the proof of this proposition. 

\end{proof}
\begin{lem} \label{loc2a}
If $\xi'$ is a non-peripheral limit point, then there is a geodesic $\hat{c}$ in $\hat{\Gamma}$ connecting $h_0$ and $f(\xi')$ that does not meet $D$.
\end{lem}

\begin{proof}

Let $(c,\hat{c}$) be an $(\epsilon, A,r)$--nice pair of $\alpha'$ such that $c$ and $\hat{c}$ have the same initial point $h_0$. Thus, $f(\xi')=f\bigl([\alpha']\bigr)=[\hat{c}]$ and $\hat{c}$ is a geodesic in $\hat{\Gamma}$ connecting $h_0$ and $f(\xi')$. We now prove that $\hat{c} \cap D =\emptyset$.

Assume for contradiction that $\hat{c} \cap D\neq \emptyset$. Thus, there is $a\in \hat{c}\cap D_1$ or $a\in\hat{c}\cap gP$ for some $gP$ such that $v_{gP}\in D_2\cap \hat{c}$. Let $u\in c$ such that $d_S(a,u)\leq A$. Let $t'$ be a positive number such that $d_X\bigl(\alpha'(t'),\Phi(u)\bigr)\leq r$. This implies that \[d_X\bigl(\Phi(a),\alpha'(t')\bigr)\leq KA+r.\] and $\alpha'(t')$ therefore lies in the $(KA+r)$--neighborhood of $Y$. We claim that $t'\geq t$. Indeed, if $t'\leq t$ then $\alpha(t')$ lies in the (KA+r+1)--neighborhood of $Y$ and this implies that $\alpha\cap N_{M_1}(Y)\neq \emptyset$. This is a contradiction for the choice of $\alpha$. Thus, $t'\geq t$.

We now show that there is a peripheral space $Y_{gP}$ such that $v_{gP}\in D_2$ and $\diam\bigl(N_{1+R_1}(Y_{g_0P_0})\cap N_{M_2+1}(Y_{gP})\bigr)$ is greater than or equal $t$. This fact would contradict our choice of $t$ and we finish the proof of the lemma. Since the distance between $\alpha'(t')$, $\Phi(h_0)$ is $t'$ and the distance between $\alpha'(t')$, $\Phi(a)$ is at most $KA+r$, then the distance between $\Phi(h_0)$, $\Phi(a)$ is at least $t'-KA+r$. Also $t'-KA+r\geq t-KA+r>\ell_1$. Thus, distance between $\Phi(h_0)$, $\Phi(a)$ is greater than $\ell_1$. This implies that $a\in\hat{c}\cap gP$ for some $gP$ such that $v_{gP}\in D_2$. Since $\alpha'(t')$, $\alpha'(0')$ both lie in the $\ell$-neighborhood of $Y_gP$, then $\alpha'\bigl([0,t']\bigr)$ lies in the $M_1$--neighborhood of $Y_gP$. In particular, $\alpha'\bigl([0,t]\bigr)$ lies in the $M_1$--neighborhood of $Y_gP$. Obviously, $\alpha'([0,t])$ lies in the $(1+R_1)$--neighborhood of $Y_{g_0P_0}$. Thus, $\diam\bigl(N_{1+R_1}(Y_{g_0P_0})\cap N_{M_2+1}(Y_{gP})\bigr)$ is greater than or equal $t$. This fact contradicts our choice of $t$. Therefore, $\hat{c} \cap D=\emptyset$.
\end{proof}

\begin{lem} \label{loc2b}
If $\xi'$ is a non-peripheral limit point, then there is a geodesic $\hat{c}$ in $\hat{\Gamma}$ connecting $h_0$ and $f(\xi')$ that does not meet $D$.
\end{lem}

\begin{proof}
Suppose that $\xi'$ is an element in $\partial Y_{g'_0P'_0}$ for some peripheral left coset $g'_0P'_0$. Thus, $f(\xi')=v_{g'_0P'_0}$. Let $\hat{c}$ be a geodesic connecting $h_0$ and $v_{g'_0P'_0}$. We now prove that $\hat{c} \cap D =\emptyset$.

Let $c$ be an $\epsilon$--quasigeodesic ray with the initial point $h_0$ in $\Gamma(G,S)$ such that the Hausdorff distance between $\Phi(c)$ and $\alpha'$ is at most $r$. Also $\alpha'$ lies in some $R'$--neighborhood of $Y_{g'_0P'_0}$. Thus, $\Phi(c)$ lies in the $(R'+r)$--neighborhood of $Y_{g'_0P'_0}$. This implies that $c$ lies in the $\bigl(K(R'+r)+K\bigr)$--neighborhood of $g'_0P'_0$. Therefore, each $G$--vertex of $\hat{c}$ lies in the $A$--neighborhood of $c$. We argue as in Lemma \ref{loc2a} to conclude $\hat{c} \cap D =\emptyset$.
\end{proof}

\begin{lem}
The map $f$ is a quotient map.
\end{lem}
\begin{proof}
This is obvious since $f$ is continuous, $\partial X$ is a compact space and $\Delta_{\infty}(\hat{\Gamma})$ is Hausdorff.
\end{proof} 

From all the above lemmas in this section, the following lemma is obvious and completes the proof of the main theorem.
\begin{lem}
The map $f$ induces a $G$--equivariant homeomorphism from the space obtained from $\partial X$ by identifying the peripheral limit points of the same type to $\Delta_{\infty}(\hat{\Gamma})$. 
\end{lem}
%%%%%%%%%%%%%%%%%%%%%%%%%%%%%%%%%%%%%%%%%%%%%%%%%%%%%%%%%%%
%%                BIBLIOGRAPHY
%%%%%%%%%%%%%%%%%%%%%%%%%%%%%%%%%%%%%%%%%%%%%%%%%%%%%%%%%%%
\bibliographystyle{alpha}
\bibliography{Tran}

\begin{thebibliography}{GdlH90}

\bibitem[Bal95]{MR1377265}
Werner Ballmann.
\newblock {\em Lectures on spaces of nonpositive curvature}, volume~25 of {\em
  DMV Seminar}.
\newblock Birkh\"auser Verlag, Basel, 1995.
\newblock With an appendix by Misha Brin.

\bibitem[BH99]{MR1744486}
Martin~R. Bridson and Andr{\'e} Haefliger.
\newblock {\em Metric spaces of non-positive curvature}, volume 319 of {\em
  Grundlehren der Mathematischen Wissenschaften [Fundamental Principles of
  Mathematical Sciences]}.
\newblock Springer-Verlag, Berlin, 1999.

\bibitem[Bow99]{MR1680044}
B.~H. Bowditch.
\newblock Boundaries of geometrically finite groups.
\newblock {\em Math. Z.}, 230(3):509--527, 1999.

\bibitem[Bow01]{MR1837220}
B.~H. Bowditch.
\newblock Peripheral splittings of groups.
\newblock {\em Trans. Amer. Math. Soc.}, 353(10):4057--4082, 2001.

\bibitem[Bow12]{MR2922380}
B.~H. Bowditch.
\newblock Relatively hyperbolic groups.
\newblock {\em Internat. J. Algebra Comput.}, 22(3):1250016, 66, 2012.

\bibitem[Cap09]{MR2665193}
Pierre-Emmanuel Caprace.
\newblock Buildings with isolated subspaces and relatively hyperbolic {C}oxeter
  groups.
\newblock {\em Innov. Incidence Geom.}, 10:15--31, 2009.

\bibitem[DS05]{MR2153979}
Cornelia Dru{\c{t}}u and Mark Sapir.
\newblock Tree-graded spaces and asymptotic cones of groups.
\newblock {\em Topology}, 44(5):959--1058, 2005.
\newblock With an appendix by Denis Osin and Sapir.

\bibitem[Far98]{MR1650094}
B.~Farb.
\newblock Relatively hyperbolic groups.
\newblock {\em Geom. Funct. Anal.}, 8(5):810--840, 1998.

\bibitem[GdlH90]{MR1086648}
{\'E}.~Ghys and P.~de~la Harpe, editors.
\newblock {\em Sur les groupes hyperboliques d'apr\`es {M}ikhael {G}romov},
  volume~83 of {\em Progress in Mathematics}.
\newblock Birkh\"auser Boston Inc., Boston, MA, 1990.
\newblock Papers from the Swiss Seminar on Hyperbolic Groups held in Bern,
  1988.

\bibitem[Ger12]{MR2989436}
Victor Gerasimov.
\newblock Floyd maps for relatively hyperbolic groups.
\newblock {\em Geom. Funct. Anal.}, 22(5):1361--1399, 2012.

\bibitem[GP]{GerasimovPotyagailo}
Victor Gerasimov and Leonid Potyagailo.
\newblock Quasi-isometric maps and {F}loyd boundaries of relatively hyperbolic
  groups.
\newblock Preprint. arXiv:0908.0705.

\bibitem[Gro87]{MR919829}
M.~Gromov.
\newblock Hyperbolic groups.
\newblock In {\em Essays in group theory}, volume~8 of {\em Math. Sci. Res.
  Inst. Publ.}, pages 75--263. Springer, New York, 1987.

\bibitem[Hin05]{MR2172938}
Mohamad~A. Hindawi.
\newblock Large scale geometry of 4-dimensional closed nonpositively curved
  real analytic manifolds.
\newblock {\em Int. Math. Res. Not.}, (30):1803--1815, 2005.

\bibitem[HK05]{MR2175151}
G.~Christopher Hruska and Bruce Kleiner.
\newblock Hadamard spaces with isolated flats.
\newblock {\em Geom. Topol.}, 9:1501--1538, 2005.
\newblock With an appendix by the authors and Mohamad Hindawi.

\bibitem[Hru10]{Hruska10}
G.C. Hruska.
\newblock Relative hyperbolicity and relative quasiconvexity for countable
  groups.
\newblock {\em Algebr. Geom. Topol.}, 10(3):1807--1856, 2010.

\bibitem[KL95]{MR1339818}
M.~Kapovich and B.~Leeb.
\newblock On asymptotic cones and quasi-isometry classes of fundamental groups
  of {$3$}-manifolds.
\newblock {\em Geom. Funct. Anal.}, 5(3):582--603, 1995.

\bibitem[MOY]{MatsudaOguniYamagata}
Y.~Matsuda, S.~Oguni, and S.~Yamagata.
\newblock Blowing up and down compacta with geometrically finite convergence
  actions of a group.
\newblock Preprint. arXiv:1201.6104.

\bibitem[Osi06]{Osin06}
Denis~V. Osin.
\newblock Relatively hyperbolic groups: {I}ntrinsic geometry, algebraic
  properties, and algorithmic problems.
\newblock {\em Mem. Amer. Math. Soc.}, 179(843):1--100, 2006.

\bibitem[Rua05]{MR2182938}
Kim Ruane.
\newblock C{AT}(0) boundaries of truncated hyperbolic space.
\newblock {\em Topology Proc.}, 29(1):317--331, 2005.
\newblock Spring Topology and Dynamical Systems Conference.

\end{thebibliography}
\end{document}